\documentclass{amsart}

\usepackage{amsmath,amssymb,amsthm}
\usepackage{graphicx,tikz}
\newtheorem{theorem}{Theorem}
\newtheorem*{thm}{Theorem}

\theoremstyle{definition}

\theoremstyle{remark}

\begin{document}

\title[]{Poissonian Pair Correlation in Higher Dimensions}

\thanks{The author is supported by the NSF (DMS-1763179) and the Alfred P. Sloan Foundation.}

\author[]{Stefan Steinerberger}
\address{Department of Mathematics, Yale University, New Haven, CT 06511, USA}
\email{stefan.steinerberger@yale.edu}

\begin{abstract}  Let $(x_n)_{n=1}^{\infty}$ be a sequence on the torus $\mathbb{T}$ (normalized to length 1). A sequence $(x_n)$ is said to have Poissonian pair correlation if, for all $s>0$,
 $$ \lim_{N \rightarrow \infty}{ \frac{1}{N} \# \left\{ 1 \leq m \neq n \leq N: |x_m - x_n| \leq \frac{s}{N} \right\}} = 2s.$$ 
It is known that this implies uniform distribution of the sequence $(x_n)$. Hinrichs, Kaltenb\"ock, Larcher, Stockinger \& Ullrich extended this result to higher dimensions and showed that sequences $(x_n)$ in $[0,1]^d$ that satisfy, for all $s>0$,
 $$ \lim_{N \rightarrow \infty}{ \frac{1}{N} \# \left\{ 1 \leq m \neq n \leq N: \|x_m - x_n\|_{\infty} \leq \frac{s}{N^{1/d}} \right\}} = (2s)^d$$ 
are also uniformly distributed. We prove the same result for the extension by the Euclidean norm: if a sequence $(x_n)$ in $\mathbb{T}^d$ satisfies, for all $s > 0$,
 $$ \lim_{N \rightarrow \infty}{ \frac{1}{N} \# \left\{ 1 \leq m \neq n \leq N: \|x_m - x_n\|_{2} \leq \frac{s}{N^{1/d}} \right\}} = \omega_d s^d,$$ 
where $\omega_d$ is the volume of the unit ball, then $(x_n)$ is uniformly distributed. Our approach shows that Poissonian Pair Correlation implies
an exponential sum estimate that resembles and implies the Weyl criterion. 
\end{abstract}

\maketitle

\section{Introduction}

\subsection{Introduction.} 
 Let $(x_n)_{n=1}^{\infty}$ be a sequence on $[0,1]$. If the sequence is comprised of independent uniformly distributed random variables, then, for all $s \geq 0$,
 $$ \lim_{N \rightarrow \infty}{ \frac{1}{N} \# \left\{ 1 \leq m \neq n \leq N: |x_m - x_n| \leq \frac{s}{N} \right\}} = 2s \qquad \mbox{almost surely}$$ 
which motivates the study of sequences with this property; the property is known as Poissonian pair correlation and appears naturally in various places.\\
 Hermann Weyl has shown that if $\alpha$ is irrational,
then the fractional parts $\left\{\alpha n^d \right\}$ are uniformly distributed in $[0,1]$. In light of this, one might ask about the spacing between the elements of the sequence.
 Some of these questions are very hard. Regarding the spacing between elements of the sequence and 0, an old 1948 result of Heilbronn \cite{heil} is that for any real $\alpha$ the inequality
$$ \| n^2 \alpha \| \lesssim n^{-\frac{1}{2} + \varepsilon}$$
has infinitely many solutions (here $\| \cdot \|$ denotes distance to the nearest integer). The exponent was improved to 2/3 by Zaharescu \cite{zaharescu}. 
Returning to the original setup, a natural problem is whether $\left\{n^2 \alpha\right\}$ behaves like a Poissonian random variable in the sense of pair correlation. 
This inspired a lot of work by Boca \& Zaharescu \cite{boca}, El-Baz, Marklof \& Vinogradov \cite{el}, Heath-Brown \cite{heath}, Marklof \cite{jens1},
Nair \& Pollicott \cite{poll}, Rudnick \& Sarnak \cite{rud1}, Rudnick, Sarnak \& Zaharescu \cite{rud2}, Rudnick \& Zaharescu \cite{rud2, rud3} and Walker \cite{walker} (among others).\\

Despite a lot of activity, until recently it was not clear how uniform distribution of a sequence and the property of its gaps exhibiting a Poissonian pair correlation structure were related. One would assume that having a rigid structure dominating the behavior of the gaps should force the sequence to be uniformly distributed but this was not proven until recently by Aistleitner, Lachmann \& Pausinger \cite{aist} and, independently, Grepstad \& Larcher \cite{grep}. Moreover, their arguments are far from straightforward.
 \begin{thm}[Aistleitner-Lachmann-Pausinger \cite{aist}, Grepstad-Larcher \cite{grep}]  Let $(x_n)_{n=1}^{\infty}$ be a sequence on $[0,1]$ and assume that for all $s>0$
  $$ \lim_{N \rightarrow \infty}{ \frac{1}{N} \# \left\{ 1 \leq m \neq n \leq N: |x_m - x_n| \leq \frac{s}{N} \right\}} = 2s,$$ 
  then the sequence is uniformly distributed.
 \end{thm}
The two proofs are structurally quite different from one another. The result has also been extended to a notion of pair correlation in higher dimensions.
 \begin{thm}[Hinrichs, Kaltenb\"ock, Larcher, Stockinger \& Ullrich \cite{hin}]  Let $(x_n)_{n=1}^{\infty}$ be a sequence in $[0,1]^d$ and assume that for all $s>0$
  $$ \lim_{N \rightarrow \infty}{ \frac{1}{N} \# \left\{ 1 \leq m \neq n \leq N: \|x_m - x_n\|_{\infty} \leq \frac{s}{N^{1/d}} \right\}} = (2s)^d,$$ 
  then the sequence is uniformly distributed.
 \end{thm}
The proof of this second result makes explicit use of the structure of $\ell^{\infty}-$unit ball and does not seem to generalize to $\ell^2-$distances. We will
\begin{enumerate}
\item give a proof that Poissonian pair correlation with respect to Euclidean distance $\ell^2$ implies uniform distribution
\item give an especially simple proof of the original result in $d=1$ dimensions
\item and provide a quantitative description of the fact that Poissonian Pair Correlation is a \textit{much} stronger property than uniform distribution. We believe this to be
the most important contribution of our paper since we derive an exponential sum estimate that is closely related to Weyl's criterion.
\end{enumerate}

\section{Results} 
\subsection{Main Result.} We now state our main result: $\mathbb{T}^d$ denotes the $d-$dimensional torus scaled to have volume 1. $\omega_d$ will denote the volume of the unit sphere in $\mathbb{R}^d$. 

\begin{theorem} Let $(x_n)$ be a sequence in $\mathbb{T}^d$. If, for all $s > 0$,
  $$ \lim_{N \rightarrow \infty}{ \frac{1}{N} \# \left\{ 1 \leq m \neq n \leq N: \|x_m - x_n\|_{2} \leq \frac{s}{N^{1/d}} \right\}} = \omega_d s^d,$$ 
then $(x_n)$ is uniformly distributed.
\end{theorem}
We will actually prove a stronger result: if $(x_n)$ satisfies Poissonian pair correlation at scale $s N^{-1/d}$ for all $0 < s < t$, then
$$\boxed{ \limsup_{N \rightarrow \infty} \sum_{1 \leq \| \ell \|_2 \leq t^{-1} N^{1/d}}^{}{  \frac{1}{N^2}  \left| \sum_{k=1}^{N}{e^{2 \pi i \left\langle \ell, x_k\right\rangle}} \right|^2} \leq \frac{c_d}{t^d}, \qquad \qquad (\diamond)}$$
where $c_d$ is a constant depending only on the dimension. However, the right-hand side can be made arbitrarily small by choosing $t$ sufficiently large and
the uniform distribution then follows from Weyl's criterion. The estimate $(\diamond)$ is best possible up to the value of $c_d$.
This shows, in a quantitative sense, that Poissonian pair correlation is a \textit{much} stronger property
than uniform distribution: equidistribution merely requires that all exponential sums tend to 0, here we require that a sum over them is bounded (and small
in the sense above). This is already close to what is impossible: Montgomery's estimate shows that if we extend summation to capture at least $c_{d,2} N$ terms (as opposed to $N/t^d$), then
the sum cannot be arbitrarily small \cite{bilyk, mon1, mon2, mon3}. This is mirrored in the not so surprising fact that we could not expect any notion of Poissonian pair correlation to hold 
below the scale $N^{-1/d}$ which is the typical scale of gaps.\\

We believe this exponential sum estimate to be of quite some interest even in $d=1$ dimensions, where we will show (see Theorem 2) an explicit estimate for sequences $(x_n)$ exhibiting Poissonian Pair Correlation: for any fixed $t>0$
$$ \limsup_{N \rightarrow \infty}   \sum_{1 \leq \ell \leq (8 t)^{-1}N}^{}{  \frac{1}{N^2} \left| \sum_{k=1}^{N}{e^{2 \pi i \ell x_k}} \right|^2} \leq \frac{1}{2t}.$$
If $\left\{x_1, \dots, x_N\right\}$ is a set of i.i.d. uniformly distributed random variables, then
$$ \mathbb{E}  \sum_{1 \leq \ell \leq (8 t)^{-1}N}^{}{  \frac{1}{N^2} \left| \sum_{k=1}^{N}{e^{2 \pi i \ell x_k}} \right|^2}  =  \sum_{1 \leq \ell \leq (8 t)^{-1}N}^{}{  \frac{1}{N}} = \frac{1}{8t}$$
which shows that the bound is sharp up to a factor of at most 4.

\subsection{Extensions.}
Our proof does not distinguish between Poissonian Pair Correlation and weaker notions such as the ones introduced by Pollicot \& Nair \cite{poll}; we only a discuss this in one dimension but versions in higher dimensions can be easily obtained. We say that a
sequence $(x_n)$ on $[0,1]$ is said to have \textit{weak pair correlation} for some $0 < \alpha < 1$ if
  $$ \lim_{N \rightarrow \infty}{ \frac{1}{N^{2-\alpha}} \# \left\{ 1 \leq m \neq n \leq N: |x_m - x_n| \leq \frac{s}{N^{\alpha}} \right\}} = 2s.$$ 
Our proof of the main result covers that case as well and shows that such sequences are also uniformly distributed: the analogous exponential sum estimate that is implied
by this property is still much stronger than uniform distribution and, if it is satisfied for all $0 < s < t$, then
$$ \sum_{ 1 \leq \ell \leq t^{-1} N^{\alpha} }^{ }{  \left| \sum_{k=1}^{N}{e^{2 \pi i \ell x_k}} \right|^2} \leq c_{\alpha}\frac{N^{1+\alpha}}{t}, \qquad \qquad (\diamond_2)$$
We see that this condition interpolates between the case of Poissonian Pair Correlation $\alpha = 1$ and the case $\alpha = 0$ (that follows more or less directly from the definition of uniform distribution itself). Our approach allows for the derivation of versions of that Theorem in higher dimensions as well. 

\subsection{Values of $s$.} It has been pointed out by the authors of \cite{hin} that their proof only requires
 $$ \lim_{N \rightarrow \infty}{ \frac{1}{N} \# \left\{ 1 \leq m \neq n \leq N: \|x_m - x_n\|_{\infty} \leq \frac{s}{N^{1/d}} \right\}} = (2s)^d$$ 
for all $s \in \mathbb{N}$ as opposed to all $s>0$. Similarly, our proof only requires 
  $$ \lim_{N \rightarrow \infty}{ \frac{1}{N} \# \left\{ 1 \leq m \neq n \leq N: \|x_m - x_n\|_{2} \leq \frac{s}{N^{1/d}} \right\}} = \omega_d s^d$$ 
to be true for a discrete sequence $(s_n)_{n \in \mathbb{N}}$ satisfying
$$ \lim_{N \rightarrow \infty} \max_{1 \leq n \leq N}{ \frac{s_{n+1} - s_n}{s_{N+1}}} = 0.$$
Put differently, if we rescale $\left\{0, s_1, \dots, s_N\right\}$ to the interval $[0,1]$, then we require that the size of the maximum gap tends to 0 as $N \rightarrow \infty$. It is not clear to us whether this condition is necessary.

\subsection{Discrepancy estimates.} All of this is quite distinct from classical notions of discrepancy since we do not prescribe quantitative rates of convergence: purely random sequences in $[0,1]$ exhibit Poissonian Pair Correlation almost surely but their discrepancy is at scale $\sim N^{-1/2}$ (ignoring logarithmic factors) and they are not especially regular. However, we note that our proof suggests that imposing a notion of speed of convergence towards Poissonian (or weak) pair correlation should imply quantitative estimates on discrepancy (see also \cite{grep, stein1}).

\section{A simple proof in one dimension}
The purpose of this section is to give a simple new proof that Poissonian Pair Correlation implies the desired exponential sum estimates for $d=1$. This proof will then naturally generalize to higher dimensions; for $d=1$ it has the advantage of
being completely explicit down to the level of constants. Our main idea is the following: suppose we have pair correlation on all scales $s/N$ with $0 < s < 1/2$.
Let $f:\mathbb{R} \rightarrow \mathbb{R}$ be continuous and compactly supported in $(-1/2,1/2)$, then
$$
 \lim_{N \rightarrow \infty}{ \frac{1}{N^2}\sum_{1\leq m \neq n \leq N}{ f(N(x_m - x_n)) \chi_{|x_m - x_n| < (2N)^{-1}}}} = \int_{-1/2}^{1/2}{f(x) dx}. \quad (\diamond \diamond)
$$
This follows from the continuity of $f$ and the assumption of Poissonian Pair Correlation (and is completely equivalent, indeed, in some
papers the notion of Poissonian Pair Correlation is defined in this way, see e.g. \cite{el}). Our proof for $d=1$ is based on using this simple property for a rescaled copy of 
$$f(x) = (2-4|x|)\chi_{|x| < 1/2}.$$ 
The relevant properties of $f$ are that it is a compactly supported symmetric probability distribution all of whose Fourier coefficients are positive; any other such function could also be used, we chose this one because it is especially simple.
Our main result for $d=1$ is the following exponential sum estimate.

\begin{theorem} Let $(x_n)_{n \in \mathbb{T}}$ be a sequence satisfying
 $$ \lim_{N \rightarrow \infty}{ \frac{1}{N} \# \left\{ 1 \leq m \neq n \leq N: |x_m - x_n| \leq \frac{s}{N} \right\}} = 2s$$ 
for all $0 < s < t$. Then
$$ \limsup_{N \rightarrow \infty}   \sum_{1 \leq \ell \leq (8 t)^{-1}N}^{}{  \frac{1}{N^2} \left| \sum_{k=1}^{N}{e^{2 \pi i \ell x_k}} \right|^2} \leq \frac{1}{2t}.$$
\end{theorem}

This estimate, as a trivial corollary, shows that for every $k \in \mathbb{N}$,
$$  \limsup_{N \rightarrow \infty}  \frac{1}{N} \left| \sum_{k=1}^{N}{e^{2 \pi i \ell x_k}} \right| \leq \frac{1}{\sqrt{2t}}.$$
In the case of Poissonian Pair Correlation, we can pick $t$ arbitrarily large and thus obtain equidistribution of $(x_n)$ from
Weyl's criterion. As mentioned above, it also shows that Poissonian Pair Correlation is a much harder condition for a sequence to satisfy since all the exponential sums have to be simultaneously small.

\begin{proof}
Let us fix $0< \delta < 1/2$ and consider the function 
$$g(x) = \frac{\chi_{\left[-\delta, \delta\right]}(x)}{2\delta}.$$
This function has average value 1, we can thus compute its Fourier coefficients as $\widehat{g}(0) = 1$ and, for $k \in \mathbb{Z}$ and $k \neq 0$,
$$  \widehat g(k) =    \int_{-1/2}^{1/2}{ \frac{\chi_{|y| \leq \delta}}{2 \delta} e^{-2 \pi i k y} dy} ds = \frac{\sin{(2 k \pi \delta)}}{2 k \pi \delta}.$$
We will now work with the function $f = g * g$, where $*$ denotes convolution; we observe that
$$ \widehat{f}(k) = \widehat{g}(k)^2 = \left( \frac{\sin{(2 k \pi \delta)}}{2 k \pi \delta} \right)^2 \geq 0.$$
There is also a simple closed form
$$ f(x)=\left( \frac{1}{2\delta} - \frac{|x|}{4\delta^2}\right) \chi_{|x| \leq 2\delta}.$$
We note that $f$ strongly depends on the choice of $\delta$ (which will be at the scale at which Poissonian Pair Correlation holds).
We will now compute a relevant quantity in two different ways: firstly, we have, separating diagonal and off-diagonal terms and using $f(0) = 1/(2\delta)$,
\begin{align*}
 \sum_{m,n=1}^{N}{f(x_m - x_n)} &= N f(0) + \sum_{m,n = 1 \atop m \neq n}^{N}{f(x_m - x_n)}\\
&=\frac{N}{2\delta} + \sum_{m,n = 1 \atop m \neq n}^{N}{f(x_m - x_n)}.
\end{align*}
If the sequence satisfies Poissonian asymptotics for all scales up to scale $\delta$, then the second summand simplifies dramatically since, using the alternative definition of Poissonian Pair Correlation ($\diamond \diamond$),
\begin{align*}
\lim_{N \rightarrow \infty} \frac{1}{N^2} \sum_{m,n = 1 \atop m \neq n}^{N}{f(x_m - x_n)} =  \int_{-1/2}^{1/2}{ f(x) dx} = 1. \qquad (\diamond \diamond \diamond)
\end{align*}
However, we can also rewrite the sum as an inner product of two measures and then use the Plancherel identity 
$$\left\langle f, g \right\rangle = \left\langle \widehat{f}, \widehat{g} \right\rangle$$
together with $\widehat{f}(0) = 1$ to compute
\begin{align*}
  \sum_{m,n=1}^{N}{f(x_m - x_n)}  &= \left\langle f * \left( \sum_{k=1}^{N}{ \delta_{x_k}}  \right),  \sum_{k=1}^{N}{ \delta_{x_k}}  \right\rangle \\
&=  \sum_{\ell \in \mathbb{Z}}{ \widehat{f}(\ell)  \left| \sum_{k=1}^{N}{e^{-2 \pi i \ell x_k}} \right|^2} \\
&= N^2 + 2\sum_{\ell =1}^{\infty}{   \left( \frac{\sin{(2 \ell \pi \delta)}}{2 \ell \pi \delta} \right)^2 \left| \sum_{k=1}^{N}{e^{2 \pi i \ell x_k}} \right|^2}
\end{align*}
This multiplier can be bounded from below: we note that for $\ell \delta \leq 1/8$, we have
$$ \left( \frac{\sin{(2 \ell \pi \delta)}}{2 \ell \pi \delta} \right)^2 \geq \frac{1}{2}$$
and thus
$$  \sum_{m,n=1}^{N}{f(x_m - x_n)}  \geq N^2 +  \sum_{1 \leq \ell \leq (8\delta)^{-1}}^{}{   \left| \sum_{k=1}^{N}{e^{2 \pi i \ell x_k}} \right|^2}.$$
Altogether, this implies
$$ N^2 +  \sum_{1 \leq \ell \leq (8\delta)^{-1}}^{}{   \left| \sum_{k=1}^{N}{e^{2 \pi i \ell x_k}} \right|^2} \leq \frac{N}{2\delta} + \sum_{m,n=1 \atop m \neq n}^{N}{f(x_m - x_n)}.$$
Setting $\delta = t/N$, dividing by $N^2$, letting $N \rightarrow \infty$ and using $(\diamond \diamond \diamond)$ gives 
$$ \limsup_{N \rightarrow \infty} \frac{1}{N^2} \sum_{1 \leq \ell \leq (8\delta)^{-1}}^{}{   \left| \sum_{k=1}^{N}{e^{2 \pi i \ell x_k}} \right|^2} \leq \frac{1}{2t}$$
which is the desired result.
\end{proof}

The same proof also shows that weak pair correlation implies $(\diamond_2)$: we simply set $\delta = t/N^{\alpha}$, the rest of the argument is identical.

\section{The general case}
We now establish our main result for dimensions $d \geq 1$.
\begin{theorem} Let $(x_n)_{n \in \mathbb{T}}$ be a sequence satisfying
 $$ \limsup_{N \rightarrow \infty}{ \frac{1}{N} \# \left\{ 1 \leq m \neq n \leq N: \|x_m - x_n\|_2 \leq \frac{s}{N} \right\}} = \omega_d s^d$$ 
for all $0 < s < t$. Then
$$\limsup_{N \rightarrow \infty} \sum_{1 \leq \| \ell \|_2 \leq t^{-1} N^{1/d}}^{}{  \frac{1}{N^2}  \left| \sum_{k=1}^{N}{e^{2 \pi i \ell x_k}} \right|^2} \leq \frac{c_d}{t^d}$$
for some constant $c_d$ depending only on the dimension.
\end{theorem}

\begin{proof} We give two different formulations of the first part of the proof: one abstract and based on scaling and one with explicit functions for the sake of clarity (the second path would yield explicit constants but comes at the price of having to work with Bessel functions).
 Let $\phi:\mathbb{T}^d \rightarrow \mathbb{R}$ be a radial probability distribution
centered around the origin and compactly supported in the ball of radius 1/4. We define a probability distribution at scale $\delta$ via
$$ \phi_{\delta}(x) = \delta^{-d} \phi\left(\frac{x}{\delta}\right) $$
and note by basic scaling that the coefficients of the Fourier series are real and satisfy
$$ \widehat{ \phi_{\delta}}(\ell) \gtrsim_{\phi} 1 \qquad \mbox{for} \qquad \| \ell\|_2 \leq \frac{1}{10 \delta}.$$
We will then work, analogously to the proof of Theorem 2, with the function 
$$f_{\delta} = \phi_{\delta} * \phi_{\delta}$$
 which is another radial probability distribution, compactly supported in the ball of radius $\delta/2$, with
the additional property that its Fourier transform is positive everywhere since
$$ \widehat{f_{\delta}} = \phi_{\delta}^2.$$
Any such function would work; we could also give one explicitly: let us fix $0< \delta < 1/2$ and consider the function 
$$g(x) = \frac{\chi_{\|x\| \leq \delta}}{\omega_d \delta^d}.$$
This function is a probability distribution, therefore $\widehat{g}(0) = 1$. There is a precise formula for the Fourier coefficients given by
$$ \widehat{g}(k) = c_d \frac{J_{d/2}(2\pi \|k\| \delta)}{\|2 \pi k \delta\|^{d/2}},$$
where $c_d$ is a constant depending only on the dimension and $J_{d/2}$ is the Bessel function. We have that
$$ \frac{J_{d/2}(x)}{|x|^{d/2}} \qquad \mbox{is continuous and positive around the origin}.$$
This tells us that
$$ \widehat{g}(k) \geq c_2 \qquad \mbox{for all} \qquad \|k\| \leq \frac{c_3}{\delta},$$
where $c_2, c_3 > 0$ are two positive constants depending only on the dimension.
We work again with the function $f = g * g$, where $*$ denotes convolution and observe 
$$ \widehat{f}(k) \geq c_2^2 \qquad \mbox{for all} \qquad \|k\| \leq \frac{c_3}{\delta}.$$
Moreover, $f$ has average value 1 and is compactly supported in a $2\delta-$ball.
We will, as in the proof of Theorem 2, compute the quantity
\begin{align*}
 \sum_{m,n=1}^{N}{f(x_m - x_n)} &= N f(0) + \sum_{m,n = 1 \atop m \neq n}^{N}{f(x_m - x_n)}
\end{align*}
in two different ways. We first note that
$$ f(0) = (g*g)(0) = \int_{\mathbb{T}^d}{\frac{\chi_{\|x\| \leq \delta}}{\omega_d^2 \delta^{2d}} dx} = \frac{ 1}{\omega_d \delta^{d}}.$$

If the sequence satisfies Poissonian asymptotics for all scales up to scale $2\delta$, then the same argument
as above combined with the fact that $f$ is a radial function allows us to write
\begin{align*}
\lim_{N \rightarrow \infty} \frac{1}{N^2} \sum_{m,n = 1 \atop m \neq n}^{N}{f(x_m - x_n)} &= \lim_{N \rightarrow \infty} \int_{\mathbb{T}^d}^{}{ f(x) \frac{1}{N^2}\sum_{m \neq n}{ \delta_{x_m - x_n}}dx} \\
&= \lim_{N \rightarrow \infty} \int_{0}^{\infty}{ f(\|x\|) \frac{1}{N^2}\sum_{m \neq n}{ \delta_{\|x_m - x_n\|}}dx}\\
&=  \int_{0}^{\infty}{ f(\|x\|) \omega_d d \|x\|^{d-1} dx}= \int_{\mathbb{T}^d}^{}{ f(x) dx} = 1.
\end{align*}

We rewrite the sum as in the proof of Theorem 2 and make use of the bound on Fourier coefficients to obtain that
\begin{align*}
  \sum_{m,n=1}^{N}{f(x_m - x_n)}  &= \left\langle f * \left( \sum_{k=1}^{N}{ \delta_{x_k}}  \right),  \sum_{k=1}^{N}{ \delta_{x_k}}  \right\rangle \\
&=  \sum_{\ell \in \mathbb{Z}^d}{ \widehat{f}(\ell)  \left| \sum_{k=1}^{N}{e^{2 \pi i \left\langle \ell, x_k\right\rangle}} \right|^2} \\
&= N^2 + \sum_{\ell \in \mathbb{Z}^d \atop \ell \neq 0}^{}{  \widehat{g}(\ell)^2\left| \sum_{k=1}^{N}{e^{2 \pi i  \left\langle \ell, x_k\right\rangle}} \right|^2}\\
&\geq  N^2 + c_{2}^2 \sum_{\| \ell \|_2 \leq c_3 \delta^{-1} \atop \ell \neq 0}^{}{  \left| \sum_{k=1}^{N}{e^{2 \pi i   \left\langle \ell, x_k\right\rangle}} \right|^2}.
\end{align*}
Altogether, we have thus seen that
\begin{align*}
N^2 + c_{2}^2 \sum_{\| \ell \|_2 \leq c_3 \delta^{-1} \atop \ell \neq 0}^{}{  \left| \sum_{k=1}^{N}{e^{2 \pi i   \left\langle \ell, x_k\right\rangle}} \right|^2}
&\leq   \sum_{m,n=1}^{N}{f(x_m - x_n)} \\
&\leq N f(0) +  \sum_{m,n=1 \atop m \neq n}^{N}{f(x_m - x_n)}
\end{align*}
and know that the last summand converges to $N^2$ assuming Poissonian pair correlation up to scale $\delta$.
We plug in $\delta = t N^{-1/d}$, divide by $N^2$ and let $N \rightarrow \infty$
\begin{align*}
\limsup_{N \rightarrow \infty} 1 +  \frac{c_2^2}{N^2}\sum_{1 \leq \|\ell\|_{2} \leq N^{1/d}/t}^{}{   \left| \sum_{k=1}^{N}{e^{2 \pi i  \left\langle \ell, x_k\right\rangle}} \right|^2} &\leq  \limsup_{N \rightarrow \infty}  \frac{1}{N^2} \sum_{m,n=1}^{N}{f(x_m - x_n)}\\
&=\limsup_{N \rightarrow \infty} \frac{f(0)}{N} + 1 = \frac{1}{\omega_d t^d} + 1.
\end{align*}
Altogether, this implies
$$\limsup_{N \rightarrow \infty}  \frac{c_2^2}{N^2}\sum_{1 \leq \ell \leq N^{1/d}/t}^{}{   \left| \sum_{k=1}^{N}{e^{2 \pi i   \left\langle \ell, x_k\right\rangle}} \right|^2} \leq \frac{1}{\omega_d t^d}$$
and the right-hand side can be made arbitrarily small by making $t$ sufficiently large.
\end{proof}

We see again that the choice of the scale $\delta = t N^{-1/d}$ is somewhat arbitrary; another choice leads to exponential sum estimates for weak pair correlation at larger scales.

\end{document}